\documentclass{article}
\usepackage{latexsym,amsthm,amssymb,amscd,amsmath}
\newcounter{alphthm}
\newtheorem{thm}{Theorem}

\newtheorem{cor}{Corollary}

\newcommand{\be}{\begin{equation}}
\newcommand{\ee}{\end{equation}}
\newcommand{\ben}{\begin{enumerate}}
\newcommand{\een}{\end{enumerate}}

\def\beq{\begin{equation}}
\def\eeq{\end{equation}}

\begin{document}
\title{Weakly symmetry of a class of $g$-natural metrics on tangent bundles}
\author{E. Peyghan}

\maketitle
\begin{abstract}
Considering the class $G$ of $g$-natural metrics on the tangent bundle of a Riemannian manifold $(M, g)$, it is shown
 that the flatnees for $g$ is a necessary and sufficient condition of weakly symmetry (recurrent or pseudo-symmetry) of $G$. In particular, the cases of weakly symmetric Sasakian lift metric studied by Bejan and Crasmareanu and recurrent or pseudo-symmetric Sasakian lift metric studied by Binh and Tam\'{a}ssy are obtained.      
\end{abstract}
\section{Introduction}
In \cite{TB}, Tam\'{a}ssy and Binh introduced the notation of weakly symmetric Riemannian manifold which is a stronger variant of recurrent and pseudo-symmetric manifolds. Then they studied the weak symmetries of Einstein and Sasakian manifolds in \cite{TB1}. Recent studies show that the notion of weakly symmetry has an important role in Riemannian geometry \cite{BC}-\cite{UL}. 

In \cite{BC}, Bejan and Crasmarenu considered the Sasakian lift $g^s$ to the tangent bundle of a Riemannian manifold $(M, g)$  and proved that the weakly symmetry of $g^s$ is equivalent to the flatness for $g$ and $g^s$. Indeed, they extended the result obtained by Tam\'{a}ssy and Binh \cite{BT} for recurrent and pseudo-symmetric manifolds. Moreover, in \cite{BC} the authors provided the following open problem: to extend the present result to other classes of metrics on tangent bundles. To solving of this open problem, we consider the metric $G=ag^s+bg^h+cg^v$ ($a$, $b$, $c$ are constants) which is a class of $g$-natural metrics introduced by Abbassi and Sarih in \cite{AS} and we show that $(TM, G)$ is weakly symmetric   
(recurrent or pseudo-symmetric) Riemannian manifold if and only if $(M, g)$ is flat.

\section{Preliminaries}
Let $(M, g)$ be a Riemannian manifold with dimension $n\geq3$ and $TM$ its tangent bundle. If we consider coordinate system $x=(x^i)$ on the base manifold $M$ and corresponding coordinates $(x, y)=(x^i, y^i)$ on $TM$, then the metric $g$ has the local coefficients $g_{ij}=g(\frac{\partial}{\partial x^i}, \frac{\partial}{\partial x^j})$. Let $\nabla$ be a Riemannian connection on $M$ with coefficients $\Gamma_{ij}^{k}$
where $1\leq i, j, k\leq n$. The Riemannian curvature tensor is defined
by
\begin{equation*}
R( X, Y)Z=\nabla_{X}\nabla_{Y}Z-\nabla_{Y}\nabla_{X}Z-\nabla_{[X,
Y]}Z , \ \ \ \forall X, Y, Z \in {\cal X}(M).
\end{equation*}
Let $\pi$ the natural projection from $TM$ to $M$.
 Consider $\pi_{*}: TTM \longmapsto TM$ and put
\[
ker{\pi_{*}}_{v}=\{z\in TTM|{\pi_{*}}_{v}( z)=0 \},\ \ \ \ \ \ \ \ \forall v\in
TM.
\]
Then the vertical vector bundle on $M$ is defined by $VTM =
\bigcup_{_{v\in TM}}ker{\pi_{*}}_{v}$. A \textit{horizontal distribution}
  on $TM$ is a
 complementary distribution $HTM$  for $VTM$ on $TTM$. It is clear that $HTM$ is a horizontal vector bundle.
 By definition, we have the decomposition
 \begin{equation}\label{dd}
 TTM =VTM\oplus HTM.
 \end{equation}
Using the induced coordinates $( x^{i}, y^{i})$ on $TM$, we can choose a local field of frames
 $\{\frac{\delta}{\delta x^i}, \frac{\partial}
 {\partial y^{i}}\}$ adapted to the above decomposition namely
 $\frac{\delta}{\delta x^i}\in {\cal X}( HTM)$ and
 $\frac{\partial}{\partial y^{i}}\in {\cal X}( VTM)$ are
  sections of  horizontal and vertical sub-bundles $HTM$ and $VTM$,
  defined by
 \begin{equation}
  \frac{\delta}{\delta x^i}=\frac{\partial}{\partial x^{i}}-y^a\Gamma_{ai}^j\frac{\partial}
 {\partial y^{j}}.\label{decomp2}
 \end{equation}
According to (\ref{dd}), every vector field $\widetilde{X}$ on $TM$ has the decomposition $\widetilde{X}=h\widetilde{X}+v\widetilde{X}$. Moreover, a vector field $X=X^i\frac{\partial}{\partial x^i}$ on $M$ has the vertical lift $X^v=X^i\frac{\partial}{\partial y^i}$ and the horizontal lift $X^h=X^i\frac{\delta}{\delta x^i}$.
\subsection{A class of $g$-natural metrics on tangent bundle}
Let $g$ be a Riemannian metric on a manifold $M$. The Sasaki lift $g^s$ of $g$ is defined by 
\begin{equation*}
\left\{
\begin{array}{cc}
&\hspace{-17mm}g^s_{(x, y)}(X^h, Y^h)=g_x(X, Y),\ \ \ g^s_{(x, y)}(X^h, Y^v)=0,\\\\
&\hspace{-5mm}g^s_{(x, y)}(X^v, Y^h)=0,\ \ \ \ \ \ \ \ \ \ \ \ \ g^s_{(x, y)}(X^v, Y^v)=g_x(X, Y).
\end{array}
\right.
\end{equation*}
Also, the horizontal lift $g^h$ and the vertical lift $g^v$ of $g$ are defined as follows \cite{AS}
\begin{equation*}
\left\{
\begin{array}{cc}
&\hspace{-5mm}g^h_{(x, y)}(X^h, Y^h)=0,\ \ \ \ \ \ \ \ \ \ \ \ g^h_{(x, y)}(X^h, Y^v)=g_x(X, Y),\\\\
&\hspace{-16mm}g^h_{(x, y)}(X^v, Y^h)=g_x(X, Y),\ \ \ g^h_{(x, y)}(X^v, Y^v)=0,\\
\end{array}
\right.
\end{equation*}
\begin{equation*}
\left\{
\begin{array}{cc}
&\hspace{-5mm}g^v_{(x, y)}(X^h, Y^h)=g_x(X, Y),\ \ \ g^v_{(x, y)}(X^h, Y^v)=0,\\\\
&\hspace{-5mm}g^v_{(x, y)}(X^v, Y^h)=0,\ \ \ \ \ \ \ \ \ \ \ \ \ g^v_{(x, y)}(X^v, Y^v)=0.
\end{array}
\right.
\end{equation*}
Now we consider the metric $G=ag^s+bg^h+cg^v$, where $a$, $b$, $c$ are constants. Indeed we can present $G$ as follows
\begin{equation}\label{metr}
\left\{
\begin{array}{cc}
&\hspace{-7mm}G_{(x, y)}(X^h, Y^h)=(a+c)g_x(X, Y),\ \ \ G_{(x, y)}(X^h, Y^v)=bg_x(X, Y),\\\\
&\hspace{-5mm}G_{(x, y)}(X^v, Y^h)=bg_x(X, Y),\ \ \ \ \ \ \ \ \ \ \ \ \ G_{(x, y)}(X^v, Y^v)=ag_x(X, Y).
\end{array}
\right.
\end{equation}
This metric is a class of $g$-natural metrics and it is Riemannian if and only if $a>0$ and $\alpha=a(a+c)-b^2>0$ hold. Also, for $a=1$ and $b=c=0$, the metric $G$ reduces to the Sasaki lift metric (See \cite{AS}). Let $\widetilde{\nabla}$ be the Levi-Civita connection of $G$. Then it is characterized by \cite{AS}
\begin{equation*}
\left\{
\begin{array}{cc}
&\hspace{-12mm}(\widetilde{\nabla}_{X^h}Y^h)|_t=(\nabla_XY)^h|_t+(A(t, X, Y))^h+(B(t, X, Y))^v,\\
&\hspace{-12mm}(\widetilde{\nabla}_{X^h}Y^v)|_t=(\nabla_XY)^v|_t+(C(t, X, Y))^h+(D(t, X, Y))^v,\\
&\hspace{-6mm}(\widetilde{\nabla}_{X^v}Y^h)|_t=(C(t, Y, X))^h+(D(t, Y, X))^v,\ \ (\widetilde{\nabla}_{X^v}Y^v)|_t=0,
\end{array}
\right.
\end{equation*}
for all vector fields $X$, $Y$ on $M$, where $A$, $B$, $C$, $D$ are the tensor fields of type (1, 2) on $M$ defined by 
\begin{align*}
A(t, X, Y)&=-\frac{ab}{2\alpha}[R(X, t)Y+R(Y, t)X],\\
B(t, X, Y)&=\frac{b^2}{\alpha}R(X, t)Y-\frac{a(a+c)}{2\alpha}R(X, Y)t,\\
C(t, X, Y)&=-\frac{a^2}{2\alpha}R(Y, t)X,\ \ \ D(t, X, Y)=\frac{ab}{2\alpha}R(Y,t)X,
\end{align*}
where $t$ is thought as a vector field on $M$ with local expression $t=y^i\frac{\partial}{\partial x^i}$. Moreover, $t^v=y^i\frac{\partial}{\partial y^i}$ is the Liouville vector field and $t^h=y^i\frac{\delta}{\delta x^i}$ is the geodesic spray of the metric $g$.
\begin{thm}\label{TH}
Let $(M, g)$ be a Riemannian manifold and $G$ be the Riemannian metric given by (\ref{metr}) on $TM$ . Then the Riemannian curvature tensor $\widetilde{R}$  of $(TM, G)$  is completely determined by
\begin{align*}
\widetilde{R}(X^v, Y^v)Z^v&=0,\\
\widetilde{R}(X^v, Y^v)Z^h&=\{\frac{a^2}{\alpha}R(X, Y)Z+\frac{a^2}{4\alpha^2}[R(X, t)R(Y, t)Z-R(Y, t)R(X, t)Z]\}^h\\
&\ \ +\{\frac{ab}{\alpha}R(Y, X)Z+\frac{a^3b}{4\alpha^2}[R(Y, t)R(X, t)Z-R(X, t)R(Y, t)Z]\}^v,
\end{align*}
\begin{align*}
\widetilde{R}(X^h, Y^v)Z^v&=\{\frac{a^2}{2\alpha}R(Z, Y)X-\frac{a^4}{4\alpha^2}R(Y, t)R(Z, t)X\}^h\\
&\ \ +\{\frac{a^3b}{4\alpha^2}R(Y, t)R(Z, t)X-\frac{ab}{2\alpha}R(Z, Y)X\}^v,
\end{align*}
\begin{align*}
\widetilde{R}(X^h, Y^h)Z^v&=\{\frac{a^2}{2\alpha}[(\nabla_YR)(Z, t)X-(\nabla_XR)(Z, t)Y]\\
&\ \ +\frac{a^3b}{4\alpha^2}[R(X, t)R(Z, t)Y-R(Y, t)R(Z, t)X]\}^h\\
&\ \ +\{R(X, Y)Z+\frac{ab}{2\alpha}[(\nabla_XR)(Z, t)Y-(\nabla_YR)(Z, t)X]\\
&\ \ +\frac{a^2}{4\alpha}[R(X, R(Z, t)Y)t-R(Y, R(Z, t)X)t]\\
&\ \ +\frac{a^2b^2}{4\alpha^2}[R(Y, t)R(Z, t)X-R(X, t)R(Z, t)Y]\}^v,
\end{align*}
\begin{align*}
\widetilde{R}(X^h, Y^h)Z^h&=\{R(X, Y)Z+\frac{ab}{2\alpha}[2(\nabla_tR)(X, Y)Z-(\nabla_ZR)(X, Y)t]\\
&\ \ +\frac{a^2}{4\alpha}[R(R(Y, Z)t, t)X-R(R(X, Z)t, t)Y]\\
&\ \ +\frac{a^2b^2}{4\alpha^2}[R(X, t)R(Y, t)Z+R(X, t)R(Z, t)Y\\
&\ \ -R(Y, t)R(X, t)Z-R(Y, t)R(Z, t)X]-\frac{a^2}{2\alpha}R(R(X, Y)t, t)Z\}^h\\
&\ \ +\{-\frac{b^2}{\alpha}(\nabla_t R)(X, Y)Z+\frac{a(a+c)}{2\alpha}(\nabla_ZR)(X, Y)t\\
&\ \ +\frac{ab^3}{2\alpha^2}[R(R(Y, t)Z, X)t-R(X, t)R(Z, t)Y\\
&\ \ -R(R(X, t)Z, Y)t+R(Y, t)R(Z, t)X]\\
&\ \ +\frac{a^2b(a+c)}{4\alpha^2}[R(X, R(Y, t)Z)t+R(X, R(Z, t)Y)t\\
&\ \ -R(Y, R(X, t)Z)t-R(Y, R(Z, t)X)t\\
&\ \ -R(R(Y, Z)t, t)X+R(R(X, Z)t, t)Y]\\
&\ \ +\frac{ab}{2\alpha}R(R(X, Y)t, t)Z\}^v,
\end{align*}
\begin{align*}
\widetilde{R}(X^h, Y^v)Z^h&=\{-\frac{a^2}{2\alpha}(\nabla_XR)(Y, t)Z+\frac{a^3b}{4\alpha^2}[R(X, t)R(Y, t)Z\\
&\ \ -R(Y, t)R(Z, t)X-R(Y, t)R(X, t)Z]+\frac{ab}{2\alpha}[R(X, Y)Z\\
&\ \ +R(Z, Y)X]\}^h+\{\frac{ab}{2\alpha}(\nabla_XR)(Y, t)Z-\frac{a^2b^2}{4\alpha^2}[R(X, t)R(Y, t)Z\\
&\ \ -R(Y, t)R(Z, t)X-R(Y, t)R(X, t)Z]+\frac{a^2}{4\alpha}R(X, R(Y, t)Z)t\\
&\ \ -\frac{b^2}{\alpha}R(X, Y)Z+\frac{a(a+c)}{2\alpha}R(X, Z)Y\}^v.
\end{align*}
\end{thm}
\begin{proof}
The proof is an special case of the proof of Proposition 2.9 of \cite{AS}.
\end{proof}
\section{Weakly symmetric Riemannian manifold $(TM, G)$}
Let $(M, g)$ be a Riemannian manifold. If there exist 1-forms $\alpha_1$, $\alpha_2$, $\alpha_3$, $\alpha_4$ and a vector field $A$ on $M$ such that
\begin{align*}
(\nabla_WR)(X, Y, Z)&=\alpha_1(W)R(X, Y)Z+\alpha_2(X)R(W, Y)Z+\alpha_3(Y)R(X, W)Z\nonumber\\
&\ \ \ +\alpha_4(Z)R(X, Y)W+g(R(X, Y)Z, W)A,
\end{align*}
then $(M, g)$ is called weakly symmetric. In \cite{DB}, the authors proved that the relations $\alpha_2=\alpha_3=\alpha_4$ and $A_2=(\alpha_2)^\sharp$ are necessary conditions to weakly symmetry of $g$. Thus a weakly symmetric manifold $(M, g)$ is characterized by:
\begin{align}\label{weak}
(\nabla_WR)(X, Y, Z)&=\alpha_1(W)R(X, Y)Z+\alpha_2(X)R(W, Y)Z+\alpha_2(Y)R(X, W)Z\nonumber\\
&\ \ \ +\alpha_2(Z)R(X, Y)W+g(R(X, Y)Z, W)(\alpha_2)^\sharp.
\end{align}
\begin{thm}
Let $(M, g)$ be a Riemannian manifold and $TM$ be its tangent bundle with Riemannian metric $G$ given by (\ref{metr}). Then $(TM, G)$ is weakly symmetric if and only if $(M, g)$ is flat. Hence $(TM, G)$ is flat.
\end{thm}
\begin{proof}
If $R=0$, then from Theorem \ref{TH}, we conclude that $\widetilde{R}=0$ and so we have (\ref{weak}). Now let $(TM, G)$ be a weakly symmetric manifold. Then we have (\ref{weak}) for all vector fields $\widetilde{X}$, $\widetilde{Y}$, $\widetilde{Z}$ and $\widetilde{W}$ on $TM$. If we suppose $\widetilde{X}=X^h$, $\widetilde{Y}=Y^v$, $\widetilde{Z}=Z^v$ and $\widetilde{W}=W^h$, then the right side of (\ref{weak}) has the following vertical part
\begin{align}\label{5}
&v\{\alpha_1(W^h)\widetilde{R}(X^h, Y^v)Z^v+\alpha_2(X^h)\widetilde{R}(W^h, Y^v)Z^v\nonumber\\
&+\alpha_2(Y^v)\widetilde{R}(X^h, W^h)Z^v+\alpha_2(Z^v)R(X^h, Y^v)W^h\nonumber\\
&+G(R(X^h, Y^v)Z^v, W^h)(\alpha_2)^\sharp\}=\{\alpha_1(W^h)[\frac{a^3b}{4\alpha^2}R(Y, t)R(Z, t)X\nonumber\\
&-\frac{ab}{2\alpha}R(Z, Y)X]+\alpha_2(X^h)[\frac{a^3b}{4\alpha^2}R(Y, t)R(Z, t)W-\frac{ab}{2\alpha}R(Z, Y)W]\nonumber\\
&+\alpha_2(Y^v)\{R(X, W)Z+\frac{ab}{2\alpha}[(\nabla_XR)(Z, t)W-(\nabla_WR)(Z, t)X]\nonumber\\
&+\frac{a^2}{4\alpha}[R(X, R(Z, t)W)t-R(W, R(Z, t)X)t]+\frac{a^2b^2}{4\alpha^2}[R(W, t)R(Z, t)X\nonumber\\
&-R(X, t)R(Z, t)W]\}+\alpha_2(Z^v)\{\frac{ab}{2\alpha}(\nabla_XR)(Y, t)W-\frac{b^2}{\alpha}R(X, Y)W\nonumber\\
&-\frac{a^2b^2}{4\alpha^2}[R(X, t)R(Y, t)W-R(Y, t)R(W, t)X-R(Y, t)R(X, t)W]\nonumber\\
&+\frac{a^2}{4\alpha}R(X, R(Y, t)W)t+\frac{a(a+c)}{2\alpha}R(X, W)Y\}\nonumber\\
&+(a+c)[-\frac{a^4}{4\alpha^2}g(R(Y, t)R(Z, t)X, W)+\frac{a^2}{2\alpha}g(R(Z, Y)X, W)]{\alpha_2}^\sharp\nonumber\\
&+b[\frac{a^3b}{4\alpha^2}g(R(Y, u)R(Z, u)X, W)-\frac{ab}{2\alpha}g(R(Z, Y)X, W)]{\alpha_2}^\sharp\}^v.
\end{align}
Now, we compute the vertical part of the left side of (\ref{weak}). Using Theorem \ref{TH} we obtain
\begin{align}\label{1}
v(\widetilde{\nabla}_{W^h}\widetilde{R}(X^h, Y^v)Z^v)&=\{\frac{a^3b}{4\alpha^2}\nabla_W(R(Y, t)R(Z, t)X)-\frac{ab}{2\alpha}\nabla_WR(Z, Y)X\nonumber\\
&\hspace{-2.4cm}+\frac{a^5(a+c)}{8\alpha^3}R(W, R(Y, t)R(Z, t)X)t+\frac{a^3(a+c)}{4\alpha^2}R(W, R(Z, Y)X)t\nonumber\\
&\hspace{-2.4cm}-\frac{a^4b^2}{4\alpha^3}R(W, t)R(Y, t)R(Z, t)X-\frac{a^2b^2}{2\alpha^2}R(W, t)R(Z, Y)X\nonumber\\
&\hspace{-2.4cm}+\frac{a^4b^2}{8\alpha^3}R(R(Y, t)R(Z, t)X, t)W-\frac{a^2b^2}{4\alpha^2}R(R(Z, Y)X, t)W\}^v,
\end{align}
\begin{align}\label{2}
v(\widetilde{R}(\widetilde{\nabla}_{W^h}X^h, Y^v)Z^v)&=\{\frac{a^3b}{4\alpha^2}R(Y, t)R(Z, t)\nabla_WX-\frac{ab}{2\alpha}R(Z, Y)\nabla_WX\nonumber\\
&\hspace{-2.4cm}+\frac{a^3}{4\alpha^2}R(Y, t)R(Z, t)A(t, W, X)-\frac{ab}{2\alpha}R(Z, Y)A(t, W, X)\}^v,
\end{align}
\begin{align}\label{3}
v(\widetilde{R}(X^h, \widetilde{\nabla}_{W^h}Y^v)Z^v)&=\{\frac{ab}{2\alpha}[(\nabla_XR)(Z, t)C(t, W, Y)-(\nabla_{C(t, W, Y)}R)(Z, t)X]\nonumber\\
&\hspace{-2.4cm}+\frac{a^3b}{4\alpha^2}R(\nabla_WY, t)R(Z, t)X-\frac{ab}{2\alpha}R(Z, \nabla_WY)X+R(X, C(t, W, Y))Z\nonumber\\
&\hspace{-2.4cm}+\frac{a^2b^2}{4\alpha^2}[R(C(t, W, Y), t)R(Z, t)X-R(X, t)R(Z, t)C(t, W, Y)]\nonumber\\
&\hspace{-2.4cm}+\frac{a^2}{4\alpha}[R(X, R(Z, t)C(t, W, Y))t-R(C(t, W, Y), R(Z, t)X)t]\nonumber\\
&\hspace{-2.4cm}+\frac{a^3b}{4\alpha^2}R(D(t, W, Y), t)R(Z, t)X-\frac{ab}{2\alpha}R(Z, D(t, W)Y)X\}^v,
\end{align}
\begin{align}\label{4}
v(\widetilde{R}(X^h, Y^v)\widetilde{\nabla}_{W^h}Z^v)&=\{\frac{a^3b}{4\alpha^2}R(Y, t)R(\nabla_WZ, t)X-\frac{ab}{2\alpha}R(\nabla_WZ, Y)X\nonumber\\
&\hspace{-2.4cm}+\frac{a^3b}{4\alpha^2}R(Y, t)R(D(t, W, Z), t)X+\frac{ab}{2\alpha}(\nabla_XR)(Y, t)C(u, W, Z)\nonumber\\
&\hspace{-2.4cm}-\frac{ab}{2\alpha}R(D(t, W, Z), Y)X+\frac{a^2}{4\alpha}R(X, R(Y, t)C(t, W, Z))t\nonumber\\
&\hspace{-2.4cm}-\frac{b^2}{\alpha}R(X, Y)C(t, W, Z)+\frac{a(a+c)}{2\alpha}R(X, C(t, W, Z))Y\nonumber\\
&\hspace{-2.4cm}-\frac{a^2b^2}{4\alpha^2}[R(X, t)R(Y, t)C(t, W, Z)-R(Y, t)R(C(t, W, Z), t)X\nonumber\\
&\hspace{-2.4cm}-R(Y, t)R(X, t)C(t, W, Z)]\}^v.
\end{align}
Using (\ref{1})-(\ref{4}) we have $v((\widetilde{\nabla}_{W^H}\widetilde{R})(X^H, Y^V)Z^V)$. Now we consider the following
\begin{equation}\label{IM}
v((\widetilde{\nabla}_{W^H}\widetilde{R})(X^H, Y^V)Z^V)=(\ref{5}).
\end{equation}
Setting $Y=t$ in the above equation implies
\begin{align}\label{3.6}
&-\frac{ab}{2\alpha}\alpha_1(W^h)R(Z, t)X-\frac{ab}{2\alpha}\alpha_2(X^h)R(Z, t)W+\alpha_2(t^v)\{R(X, W)Z\nonumber\\
&+\frac{ab}{2\alpha}[(\nabla_XR)(Z, t)W-(\nabla_WR)(Z, t)X]+\frac{a^2}{4\alpha}[R(X, R(Z, t)W)t\nonumber\\
&-R(W, R(Z, t)X)t]+\frac{a^2b^2}{4\alpha^2}[R(W, t)R(Z, t)X-R(X, t)R(Z, t)W]\}\nonumber\\
&+\alpha_2(Z^v)[\frac{a(a+c)}{2\alpha}R(X, W)t-\frac{b^2}{\alpha}R(X, t)W]+\frac{a}{2}g(R(Z, t)X, W)\alpha_2^\sharp\nonumber\\
&=\frac{a^2b^2}{2\alpha^2}R(W, t)R(Z, t)X-\frac{a^3(a+c)}{4\alpha^2}R(W, R(Z, t)X)t\nonumber\\
&-\frac{ab}{2\alpha}(\nabla_WR)(Z, t)X-\frac{a^2b^2}{4\alpha^2}R(R(Z, t)X, t)W+\frac{ab}{2\alpha}R(Z, t)A(t, W, X)\nonumber\\
&+\frac{ab}{2\alpha}R(D(t, W, Z), t)X-\frac{a(a+c)}{2\alpha}R(X, C(t, W, Z))t\nonumber\\
&+\frac{b^2}{\alpha}R(X, t)C(t, W, Z).
\end{align}
Similarly, setting $Z=t$ in (\ref{IM}) gives us
\begin{align}
&-\frac{ab}{2\alpha}\alpha_1(W^h)R(u, Y)X-\frac{ab}{2\alpha}\alpha_2(X^h)R(t, Y)W+\alpha_2(Y^v)R(X, W)t\nonumber\\
&+\alpha_2(t^v)\{\frac{ab}{2\alpha}(\nabla_XR)(Y, t)W-\frac{a^2b^2}{4\alpha^2}[R(X, t)R(Y, t)W-R(Y, t)R(W, t)X\nonumber\\
&-R(Y, t)R(X, t)W]+\frac{a^2}{4\alpha}R(X, R(Y, t)W)t+\frac{a(a+c)}{2\alpha}R(X, W)Y\nonumber\\
&-\frac{b^2}{\alpha}R(X, Y)W\}+\frac{a}{2}g(R(t, Y)X, W)\alpha_2^\sharp\nonumber\\
&=\frac{a^2b^2}{2\alpha^2}R(W, t)R(t, Y)X-\frac{a^3(a+c)}{4\alpha^2}R(W, R(t, Y)X)t\nonumber\\
&-\frac{ab}{2\alpha}(\nabla_WR)(t, Y)X-\frac{a^2b^2}{4\alpha^2}R(R(t, Y)X, t)W+\frac{ab}{2\alpha}R(t, Y)A(t, W, X)\nonumber\\
&+\frac{ab}{2\alpha}R(t, D(t, W, Y))X-R(X, C(t, W, Y))t.
\end{align}
Setting $Y=Z$ in the above equation and then summing it with (\ref{3.6}) derive that
\begin{align}\label{N}
& \alpha_2(Z^v)\{-\frac{b^2}{\alpha}R(X, t)W+\frac{a(a+c)+2\alpha}{2\alpha}R(X, W)t\}\nonumber\\
&+\alpha_2(t^v)\{R(X,W)Z+ \frac{ab}{2\alpha}[2(\nabla_XR)(Z, t)W-(\nabla_WR)(Z, t)X]\nonumber\\
&+\frac{a^2}{4\alpha}[2R(X, R(Z, t)W)t-R(W, R(Z, t)X)t]\nonumber\\
&+\frac{a^2b^2}{4\alpha^2}[R(W, t)R(Z, t)X-2R(X, t)R(Z, t)W+R(Z, t)R(W, t)X\nonumber\\
&+R(Z, t)R(X, t)W]-\frac{b^2}{\alpha}R(X, Z)W+\frac{a(a+c)}{2\alpha}R(X,W)Z\} \nonumber\\
&=\frac{b^2}{\alpha}R(X,t)C(t, W, Z)-\frac{a(a+c)+2\alpha}{2\alpha}R(X,C(t,W,Z))t.
\end{align}
Putting $Z=t$ in the above equation we get
\begin{equation}\label{m}
\alpha_2(t^v)[\frac{b^2}{\alpha}R(t, X)W+\frac{a(a+c)+2\alpha}{2\alpha}R(X, W)t]=0.
\end{equation}
Interchanging $X$ and $W$ in the above equation yields
\[
\alpha_2(t^v)\{\frac{b^2}{\alpha}R(t, W)X+\frac{a(a+c)+2\alpha}{2\alpha}R(W, X)t\}=0.
\]
By subtracting the above equation from (\ref{m}) we get
\[
\alpha_2(t^v)\{\frac{b^2}{\alpha}[R(t, X)W+R(W, t)X]+\frac{a(a+c)+2\alpha}{\alpha}R(X, W)t\}=0.
\]
Using Bianchi identity in above relation gives us
\[
\alpha_2(t^v)R(X, W)t=0.
\]
If $\alpha_2(t^v)\neq 0$ we have the conclusion. Now let $\alpha_2(t^v)= 0$, then from (\ref{N}) we have
\begin{align}\label{m1}
&\alpha_2(Z^v)\{\frac{a(a+c)+2\alpha}{2\alpha}R(X, W)t-\frac{b^2}{\alpha}R(X, t)W\}\nonumber\\
&=-\frac{a(a+c)+2\alpha}{2\alpha}R(X, C(t, W, Z))t+\frac{b^2}{\alpha}R(X, t)C(t, W, Z).
\end{align}
Exchanging $X$ and $W$ in the above equation we obtain
\begin{align*}
&\alpha_2(Z^v)\{\frac{a(a+c)+2\alpha}{2\alpha}R(W, X)t-\frac{b^2}{\alpha}R(W, t)X\}\nonumber\\
&=-\frac{a(a+c)+2\alpha}{2\alpha}R(W, C(t, X, Z))t+\frac{b^2}{\alpha}R(W, t)C(t, X, Z).
\end{align*}
By subtracting the above equation from (\ref{m1}) we get
\begin{align}
&\alpha_2(Z^v)\{\frac{a(a+c)+2\alpha}{\alpha}R(X, W)t+\frac{b^2}{\alpha}[R(W, t)X-R(X, t)W]\}\nonumber\\
&=\frac{a(a+c)+2\alpha}{2\alpha}[R(W, C(t, X, Z))t-R(X, C(t, W, Z))t]\nonumber\\
&+\frac{b^2}{\alpha}[R(X, t)C(t, W, Z)-R(W, t)C(t, X, Z)].
\end{align}
Using Bianchi identity in the above relation we conclude
\begin{align}
3\alpha_2(Z^v)R(X, W)t&=\frac{a(a+c)+2\alpha}{2\alpha}[R(W, C(t, X, Z))t\nonumber\\
&\hspace{-20mm}-R(X, C(t, W, Z))t]+\frac{b^2}{\alpha}[R(X, t)C(t, W, Z)\nonumber\\
&\hspace{-20mm}-R(W, t)C(t, X, Z)].
\end{align}
Now, we take the $g-$product with $t$
\begin{align}
&0=\frac{a(a+c)+2\alpha}{2\alpha}[g(R(W, C(t, X, Z))t, t)-g(R(X, C(t, W, Z))t, t)]\nonumber\\
&+\frac{b^2}{\alpha}[g(R(X, t)C(t, W, Z), t)-g(R(W, t)C(t, X, Z), t)]\nonumber\\
&=-\frac{a^2b^2}{2\alpha^2}[g(R(X, t)R(Z, t)W, t)-g(R(W, t)R(Z, t)X, t)].
\end{align}
Setting $W=t$ and $Z=X$ in the above equation we obtain
\[
0=\frac{a^2b^2}{2\alpha^2}g(R(X, t)t, R(X, t)t)
\]
If $b\neq0$ then the above equation yields $R(X, t)t=0$ which gives us $R=0$. Now let $b=0$. In this case we have $\alpha=a(a+c)$ and then from (\ref{m1}) we get
\[
\alpha_2(Z^v)R(X, W)t=-R(X, C(t, W, Z))t.
\]
Setting $W=X$ in the above equation gives us
\[
R(X, C(t, W, Z))t=0,
\]
and consequently
\[
\frac{a^2}{2\alpha}R(X, R(t, Z)X)t=0.
\]
Taking the $g-$product with $Z$ we have, $g(R(X, R(t, Z)X)t,Z)$ which gives us
\[
R(t, Z)X=0.
\]
Thus $R=0$, i.e. $(M, g)$ is flat.
\end{proof}
For $\alpha_2=0$ respectively $\alpha_1=2\alpha_2$ in (\ref{weak}) we get the following result
\begin{cor}
Let $(M, g)$ be a Riemannian manifold and $TM$ be its tangent bundle with Riemannian metric $G$ given by (\ref{metr}). Then $(TM, G)$ is recurrent or pseudo-symmetric or locally symmetric if and only if $(M, g)$ is flat. Hence $(TM, G)$ is flat.
\end{cor}
Considering $a=1$ and $b=c=0$ in (\ref{metr}) we get the results of \cite{BC}, \cite{BT} for the Sasakian lift metric $g^s$.
\begin{cor}
$(TM, g^s)$ is weakly symmetric (recurrent or pseudo-symmetric or locally symmetric) Riemannian manifold if and only if the base manifold $(M, g)$ is flat. Hence $(TM, G)$ is flat.
\end{cor}

\noindent
Esmaeil Peyghan\\
Department of Mathematics, Faculty  of Science\\
Arak University\\
Arak 38156-8-8349.  Iran\\
Email: epeyghan@gmail.com

\end{document}